\documentclass[journal,twoside,web]{ieeecolor}
\usepackage{generic}
\usepackage{bm}   
\usepackage{amsmath,amssymb,amsfonts}
\allowdisplaybreaks[4]
\usepackage{indentfirst}  
\usepackage{graphicx}  
\usepackage{subfigure}
\usepackage{algorithm}
\usepackage{algpseudocode}
\usepackage{arydshln}
\usepackage{cite}  
\usepackage{color}
\usepackage[american]{babel}
\usepackage{microtype}
\usepackage{cases}
\usepackage{enumerate}
\usepackage{mathtools}
\usepackage{epstopdf}
\usepackage{hyperref}
\hypersetup{hidelinks=true}
\usepackage{textcomp}
\def\BibTeX{{\rm B\kern-.05em{\sc i\kern-.025em b}\kern-.08em
    T\kern-.1667em\lower.7ex\hbox{E}\kern-.125emX}}
\makeatletter
\@ifundefined{IEEEproof}{\newenvironment{IEEEproof}{\proof}{\endproof}}{}
\@ifundefined{ifCLASSOPTIONcaptionsoff}{\newif\ifCLASSOPTIONcaptionsoff\CLASSOPTIONcaptionsofffalse}{}

\newtheorem{theorem}{Theorem}
\newtheorem{assumption}{Assumption}

\newtheorem{lemma}{Lemma}
\newtheorem{definition}{Definition}
\newtheorem{remark}{Remark}

\def\mT{\mathcal{T}}

\def\mN{\mathcal{N}}

\def\mbR{\mathbb{R}}

\def\mD{\mathcal{D}}

\def\mS{\mathcal{S}}
\def\mI{\mathcal{I}}

\def\Nf0{\mN(f_i^0)}
\def\Nbmf0{\mN(\bm{f}^0)}
\def\hf{\hat{f}}
\def\hx{\hat{x}}

\begin{document}
%
\title{Robust Optimization Under Objective Functional Uncertainty}
\author{Yue~Song,~\IEEEmembership{Member,~IEEE,}
          ~Yuxi~Lu,~\IEEEmembership{Member,~IEEE,}
          ~Gang~Li,~\IEEEmembership{Member,~IEEE,}
          ~Kairui~Feng,~\IEEEmembership{Member,~IEEE,}
          ~Qi~Liu,~\IEEEmembership{Member,~IEEE}
\thanks{
The authors are with the State Key Laboratory of Autonomous Intelligent Unmanned Systems, Tongji University, Shanghai 201210, China,
and also with the Shanghai Research Institute for Intelligent Autonomous Systems, Tongji University, Shanghai 201210, China.}
}

\markboth{}  
{Song \MakeLowercase{\textit{et al.}}: Robust Optimization Under Objective Functional Uncertainty}

\maketitle


\begin{abstract}
This paper proposes a new robust optimization (RO) formulation namely the RO under objective functional uncertainty (ObRO).
The ObRO adopts a min-max structure where the inner problem finds the worst-case objective function in a continuous function space to maximize the cost, and the outer problem finds the optimal decision in a Euclidean space to minimize the cost.
A solution algorithm is designed to alternately generate the worst-case objective function at the current decision and the optimal decision for the current collection of objective functions. Using operator theory, we prove that this algorithm converges to the defined ``semi-global'' saddle point of the ObRO problem.
In addition, we propose a numerical solver based on the piece-wise linearization (PWL) approximation of objective functions. The PWL approximate problem is proved to be numerically consistent with the original ObRO problem.
The obtained results are applied to the degradation-aware battery charging scheduling in distribution networks.
\end{abstract}


%
\IEEEpeerreviewmaketitle


\section{Introduction}\label{secintro}
Robust optimization (RO) serves as a fundamental decision-making tool for mitigating the impacts of ubiquitous uncertainties in the real world.
RO formulations have been evolved into a diverse family of variants with widespread engineering applications, such as state estimation \cite{ben2009robust}, robotics \cite{drnach2021robust}, traffic control \cite{ben2011robust}, power system dispatch \cite{sun2021robust}, and adversarial training \cite{sinha2018certifying}.
Representative RO models include the static single-stage RO that seeks a here-and-now decision robust to all possible realizations of uncertainty, and adaptive multi-stage RO that coordinates here-and-now decisions and recourse decisions in response to the actual realization of uncertainty \cite{zeng2013solving}.

\indent
Historically, the studies of RO have predominantly focused on uncertainties associated with the input parameters of optimization models.
To date, there are two types of RO formulations involving uncertainties embedded in the objective/constraint function expressions.
The first is distributionally robust optimization (DRO) where the objective/constraint functions contain undetermined coefficients governed by some probability distribution functions (PDFs) within an ambiguity set \cite{delage2010distributionally,zhao2018data,rahimian2022frameworks}.
DRO handles parametric uncertainty in the original problem and considers PDF-type functional uncertainty in the parameter space by leveraging statistical information of the uncertainties such as moment, divergence or optimal transport \cite{kuhn2025distributionally}.
The second stream is preference robust optimization (PRO) where the stakeholder's preference information is incomplete and its true utility function is drawn from a set of candidate functions \cite{maccheroni2002maxmin,armbruster2015decision}.
Restricted by economic principles, those candidate utility functions satisfy specific shape constraints such as concavity and monotonicity, which are leveraged in the PRO analysis  \cite{hu2017optimization,haskell2022preference}.
The functional uncertainty in PRO captures the partial unknown of a utility function as the stakeholder's behaviors are not entirely random.

\indent
Nevertheless, in some engineering problems the form of an actual function may be fully unknown due to the lack of knowledge about the underlying theory.
For instance, battery degradation is a critical issue in energy storage technologies, which is closely linked to the depth of discharge (DoD) \cite{ju2018two}.
As the electrochemical mechanism is still unclear, the relationship between battery degradation and DoD is obtained by fitting data points obtained in laboratory tests \cite{kong2024lithium,wang2026model}.
In addition, the discrepancy between laboratory tests and onsite operations of practical battery energy storage systems (BESS) introduces further noises to the degradation nature, by which we cannot presuppose any shape properties of the actual degradation function (see Fig. \ref{fig-objuncertain} for illustration).
With such an engineering requirement, it necessitates taking a step forward from the PRO and consider the RO under a full functional uncertainty.

\begin{figure}[!h]
  \centering
  \includegraphics[width=3.5in]{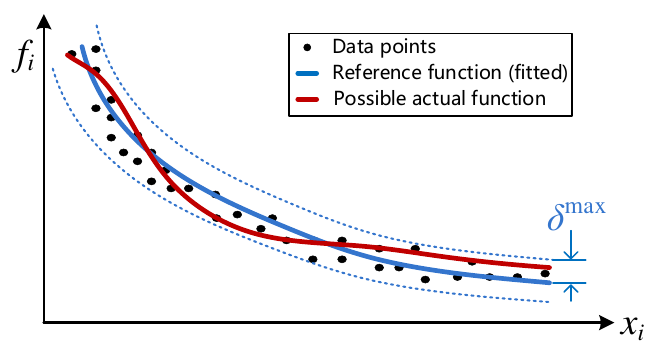}
  \caption{An example of functional uncertainty.}
  \label{fig-objuncertain}
\end{figure}

\indent
This paper investigates the RO where the objective function form is assumed completely unknown. The following threefold contributions are made.

1) We establish the minimax formulation of the robust optimization under objective functional uncertainty (ObRO), which finds the decision achieving the minimum cost under the worst-case objective function in a continuous function space (i.e., a neighborhood of a reference function).

2) Following the spirit of constraint generation approach, we design a function generation algorithm that alternately iterates between a subproblem and a master problem.
The subproblem generates the worst-case objective functions, while the master problem generates the optimal decisions. With the theory of monotonic operator, we prove that this algorithm converges to a fixed point which is a so-called semi-global saddle point of the ObRO problem.

3) Utilizing the piece-wise linearization (PWL) approximation of objective functions, we propose a numerical solver for the ObRO problem where the approximate subproblem is a linear program (LP) and the approximate master problem is a mixed-integer linear program (MILP). This PWL approximation is proved to be numerically consistent, i.e., the solution of the approximate problem approaches the solution of the original ObRO problem as PWL discretization becomes finer.
As an example application, we construct a degradation-aware BESS charging scheduling problem in distribution networks which is solved by the proposed algorithm.

\indent
The rest of the paper is organized as follows.
Section \ref{secformu} formulates the ObRO problem.
Section \ref{secalgorithm} proposes the solution algorithm with a convergence analysis.
Section \ref{secpwl} proposes the PWL numerical solver with a consistency analysis.
Section \ref{seccase} applies the method to the degradation-aware BESS charging scheduling problem.
Section \ref{secconclu} concludes the paper.

\section{Problem formulation}\label{secformu}
Let $C[a,b]$ be the space of continuous functions on interval $[a,b]$ equipped with the supremum norm
\begin{equation}
\begin{split}
      \|y\| = \max_{x\in[a,b]}|y(x)|,~\forall y\in C[a,b]
\end{split}
\end{equation}
which forms a Banach space.
Let $\bm{x}=[x_i]\in\mbR^n$ denote a vector of variables, where each variable $x_i$ takes value between $x_i^{\min}$ and $x_i^{\max}$.
With these notations, we propose the robust optimization under objective functional uncertainty (ObRO)
\begin{subequations}\label{ObRO}
\begin{align}
      &\min_{\bm{x}}\max_{f_i(\cdot)}~\sum\nolimits_{i=1}^n f_i(x_i)-\epsilon\cdot\mD_i(f_i) \label{ObRO-obj} \\
      &~~~~~~s.t.~~f_i \in\Nf0\subseteq C[x_i^{\min},x_i^{\max}],~i=1,...,n \label{ObRO-f} \\
      &~s.t.~\bm{x}\in\mS_x. \label{ObRO-x}
\end{align}
\end{subequations}
Each component of ObRO is explained in detail below.

\indent
1) Decision variable. It consists of an $n$-dimensional vector $\bm{x}$ and functions $ f_1(\cdot),f_2(\cdot),...,f_n(\cdot)$.
Hence, the ObRO problem involves optimizing the entire expression of the objective function rather than optimizing some parameters of a specific objective function.

2) Objective function. We assume the objective is separable as a sum of functions $f_i(x_i)$, which is common in engineering problems.
The form of $f_i$ is not known precisely and locates in a neighborhood of the reference function $f_i^0\in C[x_i^{\min},x_i^{\max}]$, say $\Nf0$, which will be explained later. 
In addition, $\mD_i$ is a functional of $f_i$
\begin{equation}\label{fstarmax}
\begin{split}
     \mD_i(f_i) = \int_{x_i^{\min}}^{x_i^{\max}}|f_i(x_i)-f_i^0(x_i)|~dx_i
\end{split}
\end{equation}
where $\epsilon>0$ is a small positive number.
The functional $\mD_i$ serves as a small penalty to regulate the total deviation of $f_i$ from the reference function $f_i^0$.

3) Constraint \eqref{ObRO-f}. It restricts $f_i(\cdot)$ to take expressions within $\Nf0$, which is a subset in the space of continuous functions, i.e., a neighborhood of the reference function $f_i^0$.
In this paper, for instance, we adopt the following description for $\Nf0$
\begin{subequations}\label{fneighbor}
\begin{align}
      &\Nf0 = \{f_i~\textup{satisfies} \notag \\
      &\|f_i - f_i^0\| \leq \delta_i^{\max} \label{ObRO-f0} \\
      &\mD_i(f_i) \leq d_i^{\max} \label{ObRO-totdev} \\
      &\Big|\frac{f_i(x_i^{a})-f_i(x_i^{b})}{f_i^0(x_i^{a})-f_i^0(x_i^{b})}\Big| \leq L_i,\forall x_i^{a},x_i^{b}\in[x_i^{\min},x_i^{\max}]~\}   \label{ObRO-lip}
\end{align}
\end{subequations}
where \eqref{ObRO-f0} implies that $f_i$ is in the ``$\delta_i^{\max}$-disk'' around $f_i^0$ in the sense of supremum norm;
\eqref{ObRO-totdev} sets an upper bound for the total deviation of $f_i$ from $f_i^0$;
\eqref{ObRO-lip} sets an upper bound for the rate of change of $f_i$ with $L_i>1$ being a predefined number according to the nature of the reference function $f_i^0$.
Constraints \eqref{ObRO-f0}-\eqref{ObRO-lip} ensure that $f_i$ deviates from $f_i^0$ in a realistic manner.

4) Constraint \eqref{ObRO-x}. It includes $x_i^{\min}\leq x_i\leq x_i^{\max}$, $i=1,...,n$ and possibly other restrictions on $\bm{x}\in\mbR^n$.

\indent
For convenience of analysis, we henceforth rewrite \eqref{ObRO} as
\begin{equation}\label{ObRO-compact}
\begin{split}
      \min_{\bm{x}\in\mS_x}\max_{\bm{f}\in\Nbmf0}~V(\bm{f},\bm{x})
\end{split}
\end{equation}
where $\bm{f}=[f_1(\cdot),...,f_n(\cdot)]^T$, and $V(\bm{f},\bm{x})$ and $\bm{f}\in\Nbmf0$ are the compact form of objective functional \eqref{ObRO-obj} and constraint \eqref{ObRO-f}, respectively.

\indent
The ObRO problem has a min-max structure.
Given a vector $\bm{x}$, the inner problem finds the worst-case objective function that leads to the maximum cost while not deviating from the reference function too much (regulated by the penalty $\mD_i(f_i)$).
The outer problem finds the optimal vector $\bm{x}$ that achieves the lowest cost under the worst-case uncertain objective function.
An ObRO optimum is a saddle point $(\bm{x}^*,\bm{f}^*)$ such that $\bm{x}^*$ is the minimizer of $\min_{\bm{x}\in\mS_x}~V(\bm{f}^*,\bm{x})$ and $\bm{f}^*$ is the maximizer of $\max_{\bm{f}\in\Nbmf0}~V(\bm{f},\bm{x}^*)$.
Due to the allowed non-convexity of $\bm{f}$, it is impractical to find the global saddle point.
Alternatively, it is of interest to introduce the definition below.
\begin{definition}[Semi-global saddle point]
    For the ObRO problem \eqref{ObRO-compact}, a solution $(\bm{x}^*,\bm{f}^*)$ is a semi-global saddle point if $\bm{x}^*$ is a local minimizer of $\min_{\bm{x}\in\mS_x}~V(\bm{f}^*,\bm{x})$ and $\bm{f}^*$ is a global maximizer of $\max_{\bm{f}\in\Nbmf0}~V(\bm{f},\bm{x}^*)$.
\end{definition}

\indent
A semi-global saddle point is stronger than a standard local saddle point where the outer minimization and inner maximization problems are both locally optimal.
A semi-global saddle point is more desirable in practice as it figures out the global worst-case objective and provides a locally optimal countermeasure at least.
This paper will focus on finding a semi-global saddle point of the ObRO problem.

\begin{remark}
  Due to the weak duality, it always holds
  \begin{equation*}
  \begin{split}
      \min_{\bm{x}\in\mS_x}\max_{\bm{f}\in\Nbmf0}~V(\bm{f},\bm{x})\geq \max_{\bm{f}\in\Nbmf0}\min_{\bm{x}\in\mS_x}~V(\bm{f},\bm{x}).
  \end{split}
  \end{equation*}
  But according to Sion's Minimax Theorem \cite{terkelsen1972some}, we cannot further achieve $\min\limits_{\bm{x}\in\mS_x}\max\limits_{\bm{f}\in\Nbmf0}~V(\bm{f},\bm{x})= \max\limits_{\bm{f}\in\Nbmf0}\min\limits_{\bm{x}\in\mS_x}~V(\bm{f},\bm{x})$ as $\bm{f}$ is not necessarily convex in our framework. It implies that in case of functional uncertainty, the here-and-now decision (given by the min-max problem) is more conservative than the wait-and-see decision (given by the max-min problem), the gap of which captures the cost of lacking mechanism knowledge about objective function.
\end{remark}

\begin{remark}
   There are two common approaches to solving min-max problems.
   The first approach is based on dual reformulation, i.e., to replace the inner maximization problem with its KKT condition \cite{dempe2015bilevel}, which works well when the inner problem is finite-dimensional and convex.
   The second approach is based on constraint generation, i.e., to replace the uncertainty set with a (usually finite) subset of ``worst-case realizations'' \cite{zeng2013solving,simchi2019constraint,song2025adaptive}, which works well under a linear objective function and a polyhedron-type uncertainty set.
   For the ObRO problem, the infinite-dimension and complex nature in the objective functional uncertainty necessitates a reinvestigation of the applicability and convergence of the conventional approaches.
   The next section will establish a solution method as an extension of constraint generation.
\end{remark}

\section{The solution method}\label{secalgorithm}
\subsection{The algorithm details}
For the convenience of analysis, \eqref{ObRO} is reformulated into the following equivalent form
\begin{subequations}\label{ObRO1}
\begin{align}
      \min_{\bm{x}\in\mS_x,\bm{f},\eta}~&\eta \\
      s.t.~~&\eta\geq V(\bm{f},\bm{x}),~\forall \bm{f}\in\Nbmf0  \label{ObRO1-eta}.
\end{align}
\end{subequations}
where the ObRO inner problem is transformed into the robust constraint \eqref{ObRO1-eta}.
Since \eqref{ObRO1-eta} contains an infinite number of scenarios for $\bm{f}$, it cannot be directly handled.
Inspired by the idea of constraint generation, we propose the function generation algorithm to address the difficulty in \eqref{ObRO1-eta}.
This algorithm relies on a subproblem and a master problem that generate a key subset of functions $\bm{f}$ and update solution $\bm{x}$ iteratively.

\indent
With the solution $\bm{x}^{k}$ at hand, we formulate the subproblem below to generate new functions $\bm{f}^{k+1}$
\begin{equation}\label{ObRO-sub}
\begin{split}
       \bm{f}^{k+1}=\arg\max\nolimits_{\bm{f}\in\Nbmf0}~&V(\bm{f},\bm{x}^k)
\end{split}
\end{equation}
Then, $\bm{f}^{k+1}$ are augmented into the objective function scenarios, and we formulate the master problem below to obtain an updated solution $\bm{x}^{k+1}$
\begin{equation}\label{ObRO-master}
\begin{split}
      (\bm{x}^{k+1},\eta^{k+1})=
      &\arg\min_{\bm{x}\in\mS_x,\eta}~\eta \\
      &s.t.~\eta\geq V(\bm{f}^l,\bm{x}),~l=0,1,...,k+1.
\end{split}
\end{equation}

\indent
With the above idea, we are ready to present the function generation algorithm in Algorithm~\ref{algObRO}.
In this algorithm, the subproblem \eqref{ObRO-sub} provides a new scenario of objective function $\bm{f}^{k+1}$ that is the worst case for the current solution $\bm{x}^{k}$. Meanwhile it estimates an upper bound $UB$ for the optimal cost of ObRO problem as $\bm{x}^{k}$ may be suboptimal.
Next, the master problem \eqref{ObRO-master} updates the solution $\bm{x}^{k+1}$ by considering a subset of objective function scenarios $\bm{f}^{l},l=0,1,...,k+1$ and estimates a lower bound $LB$ for the optimal cost of ObRO problem.
Also $\bm{x}^{k+1}$ is used to construct the subproblem in the next iteration.
The iteration converges in case of $UB-LB\leq \varepsilon$, implying that the current solution $\bm{x}^{k+1}$ is close enough to the very optimal one.

\indent
For finite-dimensional robust optimization problem, the convergence of constraint generation-type algorithms has been well studied \cite{lorca2016multistage,sun2021robust}.
Nevertheless, \eqref{ObRO-sub} is an infinite-dimensional functional optimization problem, where the convergence of Algorithm~\ref{algObRO} is non-trivial.
The next subsection will investigate the convergence property of Algorithm~\ref{algObRO} by means of operators.

\begin{algorithm}
\caption{}
\label{algObRO}
\begin{algorithmic}[1]
\Require{$\bm{f}^0,~\Nbmf0,~\mS_x$}
\Ensure{$\bm{x}^k,~\bm{f}^k$}
    \State Initialization: Iteration counter $k=0$. Set $\bm{f}^k=\bm{f}^0$, $UB=+\infty, LB=-\infty$.
    Solve the master problem \eqref{ObRO-master} and obtain the initial solution $\bm{x}^k$.
    \While{$UB-LB > \varepsilon$}
        \State \parbox[t]{\dimexpr\linewidth-\algorithmicindent}{
        Given $\bm{x}^{k}$, solve the subproblem \eqref{ObRO-sub} to obtain $\bm{f}^{k+1}$ and corresponding objective value $V(\bm{f}^{k+1},\bm{x}^k)$.}
        \State \parbox[t]{\dimexpr\linewidth-\algorithmicindent}{
        Update $UB\leftarrow \min\{UB,V(\bm{f}^{k+1},\bm{x}^k)\}$.}
        \State \parbox[t]{\dimexpr\linewidth-\algorithmicindent}{
        Solve the master problem \eqref{ObRO-master} to obtain $\bm{x}^{k+1}$ and corresponding objective value $\eta^{k+1}$.}
        \State Update $LB\leftarrow \eta^{k+1}$.
        \State $k\leftarrow k+1$.
    \EndWhile
\end{algorithmic}
\end{algorithm}

\subsection{Convergence analysis}
We begin the convergence analysis by regarding optimization problems \eqref{ObRO-sub} and \eqref{ObRO-master} as operators.
In fact, the subproblem \eqref{ObRO-sub} serves as an operator $\mT_1$ mapping from $\mbR^n$ to the space of continuous functions, i.e., $\bm{f}^{k+1}=\mT_1\bm{x}^k$.
Solving the master problem \eqref{ObRO-master} can be regarded as an operator $\mT_2$ mapping from the space of continuous functions to $\mbR^n$.
The composite operator $\mT_2\circ\mT_1$, which completes one round of iteration in Algorithm~\ref{algObRO}, maps from $\mbR^n$ to $\mbR^n$, i.e., $\bm{x}^{k+1}=\mT_2\circ\mT_1\bm{x}^k$.
Thus, we introduce the following concepts and assumptions on operators for our analysis.
\begin{definition}[\cite{meyer1976sufficient}]
   Consider an operator $\mT: Y_1\rightarrow Y_2$ where $Y_1, Y_2$ are topological vector spaces, then:
   \begin{itemize}
     \item $\mT$ is \textit{continuous} if for any sequence $\{y_i\}\subseteq Y_1$, $\lim_{i\rightarrow\infty} y_i=y$ implies $\lim_{i\rightarrow\infty} \mT(y_i)=\mT(y)$.
     \item When $Y_1=Y_2$, $y\in Y_1$ is a \textit{fixed point} of $\mT$ if $y=\mT(y)$.
     \item When $Y_1=Y_2$, $\mT$ is \textit{strictly monotonic} with respect to the functional $g$ if $g(\mT(y))<g(y)$ whenever $y$ is not a fixed point of $\mT$.
     \item When $Y_1=Y_2$, $\mT$ is \textit{uniformly compact} if there exists a compact set $Y_1^{\prime}\subseteq Y_1$ independent of $y$ such that $\mT(y)\in Y_1^{\prime}$ for all $y\in Y_1$.
   \end{itemize}
\end{definition}

\begin{assumption}\label{assump-unique}
   The set $\mS_x$ is closed. For any $\bm{x}^{k}\in\mS_x$, the subproblem~\eqref{ObRO-sub} has a unique globally optimum.
   For any functions $\bm{f}\in\Nbmf0$, the master problem~\eqref{ObRO-master} has a unique globally optimum.
\end{assumption}

\begin{assumption}\label{assump-bounded}
   There exists a positive number $M<\infty$ such that each reference function $f_i^0$, $i=1,...,n$ satisfies $\|f_i^0(x_i^a)\|<M$ and $\frac{|f_i^0(x_i^{a})-f_i^0(x_i^{b})|}{|x_i^{a}-x_i^{b}|}<M$ for any $x_i^{a},x_i^{b}\in[x_i^{\min},x_i^{\max}]$.
\end{assumption}

\begin{assumption}\label{assump-fixed}
   The operator $\mT_2\circ\mT_1$ has a finite number of fixed points.
\end{assumption}

\begin{remark}
   Assumption \ref{assump-unique} can be justified as follows.
   If the penalty $\mD_i(f_i)$ is absent from the objective, then \eqref{ObRO-sub} may have multiple optimal solutions as its objective $f_i(x_i)$ only focuses on the function value around $x_i$ and there is much freedom for the function expression in other regions.
   When $\mD_i(f_i)$ is involved, the objective of \eqref{ObRO-sub} contains ``global'' information of $f_i$ and attempts to maximize $\sum\nolimits_{i=1}^n f_i(x_i)$ at a low budget of total deviation of $f_i$ from $f_i^0$. In this case, it is more reasonable to assume that the optimum is unique.
   The master problem \eqref{ObRO-master} is a nonlinear program, and it is a common practice to assume the uniqueness of its global optimum.
   Assumption \ref{assump-bounded} requires the reference function to be bounded and Lipschitz continuous, which is practical.
   For Assumption \ref{assump-fixed}, it will be seen later that a fixed point of $\mT_2\circ\mT_1$ is also a semi-global saddle point of the ObRO problem (Theorem \ref{thm-fixed}), which means the number of fixed points is no greater than the number of saddle points.
   So the adoption of Assumption~\ref{assump-unique} justifies Assumption \ref{assump-fixed}.
\end{remark}

With these assumptions, the optimality and convergence of Algorithm~\ref{algObRO} are characterized as follows.

\begin{theorem}[Optimality of fixed point]\label{thm-fixed}
    Under Assumption \ref{assump-unique}, if $\bm{x}^*$ is a fixed point of the operator $\mT_2\circ\mT_1$ and $\bm{f}^*=\arg\max\nolimits_{\bm{f}\in\Nbmf0}V(\bm{f},\bm{x}^*)$,
    then $(\bm{x}^*,\bm{f}^*)$ is a semi-global saddle point of the ObRO problem \eqref{ObRO}.
\end{theorem}

\begin{IEEEproof}
    According to the preconditions, $\bm{f}^*$ is the unique global maximizer of $\max\nolimits_{\bm{f}\in\Nbmf0}V(\bm{f},\bm{x}^*)$.
    Now we prove that $\bm{x}^*$ must be a local minimizer of $\min\nolimits_{\bm{x}\in\mS_x}V(\bm{f}^*,\bm{x})$ by contradiction.

    Suppose we have $\bm{x}^{k}=\bm{x}^*$ at hand, then the next generated function is $\bm{f}^{k+1}=\mT_1\bm{x}^k=\bm{f}^*$.
    If $\bm{x}^*$ is a fixed point of $\mT_2\circ\mT_1$, then $\bm{x}^*$ is the unique optimum of the master problem below
    \begin{subequations}\label{masterfstar}
    \begin{align}
      (\bm{x}^*,\eta^*)=\arg\min_{\bm{x}\in\mS_x,\eta}~&\eta \\
      s.t.~&\eta\geq V(\bm{f}^l,\bm{x}),~l=0,1,...,k \\
      &\eta\geq V(\bm{f}^*,\bm{x}).
    \end{align}
    \end{subequations}
    Note that we have $V(\bm{f}^*,\bm{x}^*) > V(\bm{f}^l,\bm{x}^*),~l=0,1,...,k$ since $\bm{f}^*$ is the unique global maximizer.
    By the sign preserving property, there exists $\varepsilon>0$ such that
     \begin{equation}\label{fstarmax}
     \begin{split}
            V(\bm{f}^*,\bm{x}) > V(\bm{f}^l,\bm{x}),~l=0,1,...,k,~\forall \bm{x}\in\mN_{\varepsilon}(\bm{x}^*)
     \end{split}
     \end{equation}
    where $\mN_{\varepsilon}(\bm{x}^*)=\{\bm{x}:~\|\bm{x}-\bm{x}^*\|< \varepsilon\}$.
    In this case, if we assume $\bm{x}^*$ is not a local minimizer of $\min\nolimits_{\bm{x}\in\mS_x}V(\bm{f}^*,\bm{x})$, then there exists $\bm{x}^{\prime}\in\mN_{\varepsilon}(\bm{x}^*)\cap\mS_x$, $\bm{x}^{\prime}\neq \bm{x}^*$ such that
    \begin{equation}\label{fstarmax2}
     \begin{split}
       V(f^{*},\bm{x}^{\prime})\leq V(f^{*},\bm{x}^*).
     \end{split}
     \end{equation}
    By \eqref{fstarmax} and \eqref{fstarmax2}, the objective value of \eqref{masterfstar} at $\bm{x}^{\prime}$ equals $V(\bm{f}^*,\bm{x}^{\prime})\leq V(\bm{f}^*,\bm{x}^*)$, which contradicts to that $\bm{x}^*$ is the unique optimum of \eqref{masterfstar}.
    Therefore, $(\bm{x}^*, \bm{f}^*)$ is a semi-global saddle point.
\end{IEEEproof}

\begin{theorem}[Convergence to local optimum]\label{thm-converge}
   Under Assumptions \ref{assump-unique}-\ref{assump-fixed}, Algorithm~\ref{algObRO} converges to a fixed point of $\mT_2\circ\mT_1$ that is a semi-global saddle point of the ObRO problem \eqref{ObRO}.
\end{theorem}

\begin{IEEEproof}
   It can be proved by showing that $\mT_2\circ\mT_1$ is continuous, uniformly compact, and strictly monotonic with respect to $UB-LB$. The details are given in Appendix A.
\end{IEEEproof}

\indent
Theorem \ref{thm-fixed} and Theorem \ref{thm-converge} conclude the convergence of Algorithm~\ref{algObRO} to a fixed point of $\mT_2\circ\mT_1$ which is also a semi-global saddle point of the ObRO problem.
This theoretical result is good enough for practice as it is hard to guarantee the convergence to the global saddle point, which is rooted from that $\bm{f}$ is allowed to be non-convex.
In this case, a local minimizer is not equivalent to a global minimizer, and $\bm{x}^*$ not being a global minimizer of $\min\nolimits_{\bm{x}\in\mS_x}V(\bm{f}^*,\bm{x})$ does not lead to \eqref{fstarmax2} which is the key to constructing a contradiction in the proof of Theorem \ref{thm-fixed}.

\indent
Since the master problem \eqref{ObRO-master} and subproblem \eqref{ObRO-sub} in Algorithm~\ref{algObRO} usually cannot be solved analytically, the next section will design a numerical solver for them.

\section{PWL numerical solver}\label{secpwl}
The master problem \eqref{ObRO-master} is a nonlinear program with respect to the $n$-dimensional decision variable $\bm{x}$, while the subproblem \eqref{ObRO-sub} is an infinite-dimensional optimization problem with respect to the objective function $f_i$.
In general, the difficulty in \eqref{ObRO-master} stems from the nonlinearity of $f_i$ acting as an input, and the challenge in \eqref{ObRO-sub} arises from the infinite dimensionality of $f_i$ acting as a decision variable.
Therefore, the core of the numerical solver is to work out a proper finite-dimensional approximation of $f_i$ to mitigate the intractability in \eqref{ObRO-sub} and \eqref{ObRO-master}.
In this section, we propose to handle such difficulty by PWL approximation.
In addition, the property of PWL approximation also helps to establish the consistency between the original problems and corresponding approximate formulations.

\indent
To focus on handling $f_i$, in this section we assume a simple form of $\mS_x$ as a linear constraint
\begin{equation}\label{linearSx}
\begin{split}
      \mS_x = \{\bm{x}|~\bm{A}\bm{x}\leq \bm{b}\}
\end{split}
\end{equation}
where $\bm{A}\in\mbR^{m\times n}$ and $\bm{b}\in\mbR^m$.
It will be seen later that we finally obtain a tractable MILP/LP formulation by applying PWL approximation to the ObRO master/sub problem with a linear form of $\mS_x$.
The linear form of $\mS_x$ in \eqref{linearSx} can be extended to convex nonlinear versions, such as the second-order cone constraint which is common in engineering problems (e.g., see \cite{gan2015exact,kocuk2016strong,fooladivanda2018energy,zheng2019online,singh2019natural} as examples in various energy systems).
In that case, we will obtain a second-order conic program (SOCP) or mixed-integer SOCP form of the ObRO master/sub problem which is also tractable.

\indent
We now present the PWL approximation of $f_i(x_i)$.
First, the interval $[x_i^{\min},x_i^{\max}]$ is partitioned into $N-1$ segments
\begin{equation}
\begin{split}
      &\mI_N = \{[\hx_{i[1]},\hx_{i[2]}],[\hx_{i[2]},\hx_{i[3]}],\cdots,[\hx_{i[N-1]},\hx_{i[N]}]\} \\
      &\hx_{i[1]} = x_i^{\min},~\hx_{i[N]}=x_i^{\max}
\end{split}
\end{equation}
where $\hx_{i[p]}$ denotes the $p$-th sampling point. The partition is not necessarily an even one, and the size of each segment can be tailored to the shape of the reference function $f_i^0$.

\indent
For simplicity, we henceforth use $\{\hx_{i[p]}\}_{p=1}^N$ to denote the whole set of sampling points.
Accordingly, let $\{\hf_{i[p]}\}_{p=1}^N$ denote the values of $f_i(\cdot)$ at the sampling points $\{\hx_{i[p]}\}_{p=1}^N$, i.e., $\hf_{i[p]}=f_i(\hx_{i[p]})$.
By introducing $N$-1 binary variables $\{\beta_{i[p]}\}_{p=1}^{N-1}$ and $N$ continuous variables $\{\alpha_{i[p]}\}_{p=1}^N$, we have the following PWL approximation of $f_i$ 
\begin{subequations}\label{fihat}
\begin{align}
      \beta_{i[p]} &= \{0,1\},~p=1,...,N-1 \label{fihat-beta} \\
      0 &\leq \alpha_{i[p]} \leq 1,~p=1,2,...,N \\
      \sum\nolimits_{p=1}^{N-1} \beta_{i[p]} &= 1,~\sum\nolimits_{p=1}^{N} \alpha_{i[p]} = 1 \\
      \alpha_{i[p]} &\leq \beta_{i[p]} + \beta_{i[p-1]},~p=2,...,N-1 \\
      \alpha_{i[1]} &\leq \beta_{i[1]},~\alpha_{i[N]} \leq \beta_{i[N-1]} \label{fihat-alpha} \\
      \hx_i &= \sum\nolimits_{p=1}^{N} \alpha_{i[p]} \hx_{i[p]} \label{fihat-xi} \\
      \hf_{i}(\hx_i) &= \sum\nolimits_{p=1}^{N} \alpha_{i[p]} \hf_{i[p]}. \label{fihat-fi}
\end{align}
\end{subequations}
When $\beta_{i[p]}=1$, the $p$-th segment of the function is activated and \eqref{fihat-beta}-\eqref{fihat-alpha} impose that $\alpha_{i[p]},\alpha_{i[p+1]}$ are the only two possibly nonzero variables among $\{\alpha_{i[p]}\}_{p=1}^N$.
Then, for $\hx_i$ given by \eqref{fihat-xi}, the expression \eqref{fihat-fi} estimates the function value at $\hx_i$ by the linear interpolation between $\hf_{i[p]}$ and $\hf_{i[p+1]}$ (see Fig. \ref{fig-PWL} for illustration).
To sum up, using PWL approximation, the whole function $f_i(\cdot)$ is captured by a set of points $\{\hf_{i[p]}\}_{p=1}^N$ and the values of $x_i$ and $f_i(x_i)$ are estimated by the hatted notations $\hx_i, \hf_i(\hx_i)$ in \eqref{fihat}.

\indent
We now apply the PWL approximation to reformulating the subproblem \eqref{ObRO-sub}.
Assume we have $\hx_i^{k}$ at hand, which is the solution of the master problem in iteration $k$, such that $\hx_{i[p^k]} \leq \hx_i^{k}\leq \hx_{i[p^k+1]}$, i.e., $\hx_i^{k}$ locates in the $p^k$-th segment.
Then, there are only two variables in $\{\alpha_{i[p]}\}_{p=1}^N$ taking possibly nonzero values, say $\alpha_{i[p^k]},\alpha_{i[p^k+1]}$, which are determined by
\begin{equation}
\begin{split}
    \hx_i^{k}&=\alpha_{i[p^k]} \hx_{i[p^k]}+\alpha_{i[p^k+1]} \hx_{i[p^k+1]} \\
    1 &= \alpha_{i[p^k]} + \alpha_{i[p^k+1]}.
\end{split}
\end{equation}
With $\alpha_{i[p^k]},\alpha_{i[p^k+1]}$ as parameters, we have the following PWL approximation of \eqref{ObRO-sub} 
\begin{subequations}\label{ObRO-sub-PWL}
\begin{align}
     &(\{\hf^{k+1}_{i[p]}\}_{p=1}^N,~\hat{\mD}^{k+1}_i) = \notag \\
     &\arg\max~\sum\nolimits_{i=1}^n \hf_{i}(\hx_i) - \epsilon\hat{\mD}_i \\
      &\textup{over}~\{\hf_{i[p]}\}_{p=1}^N,\hat{\mD}_i,\{s_{i[p]}\}_{p=1}^N,~i=1,...,n   \notag \\
      &s.t.~|\hf_{i[p]} - \hf^0_{i[p]} | \leq \delta_i^{\max} \label{sub-PWL-delta} \\
      &~~~~~\hat{\mD}_i \leq d_i^{\max} \label{sub-PWL-d}  \\
      &~~~~~|\hf_{i[p]}-\hf_{i[p+1]}|\leq L_i|\hf^0_{i[p]}-\hf^0_{i[p+1]}|  \label{sub-PWL-lip} \\
      &~~~~~|\hf_{i[p]} - \hf^0_{i[p]}| \leq s_{i[p]} \label{sub-PWL-D1}\\
      &~~~~~\hat{\mD}_i  = \sum\nolimits_{p=1}^{N-1} \frac{1}{2}(s_{i[p+1]}+s_{i[p]}) (x_{i[p+1]}-x_{i[p]}) \label{sub-PWL-D2}  \\
      &~~~~~\hf_{i}(\hx_i) = \alpha_{i[p^k]} \hf_{i[p^k]} + \alpha_{i[p^k+1]} \hf_{i[p^k+1]}.   \label{sub-PWL-fi}
\end{align}
\end{subequations}
Note that in all constraints of \eqref{ObRO-sub-PWL}, the variable index $i$ ranges from 1 to $n$ and interval segment index $p$ ranges from 1 to $N$.
\eqref{sub-PWL-delta}-\eqref{sub-PWL-lip} are the PWL version of \eqref{ObRO-f0}-\eqref{ObRO-lip}, respectively.
In \eqref{sub-PWL-D1}-\eqref{sub-PWL-D2}, $\hat{\mD}_i$ is the trapezoidal-type numerical integration of $\mD_i$, where the auxiliary variables $\{s_{i[p]}\}_{p=1}^N$ are introduced to reformulate the absolute-value terms $|\hf_{i[p]} - \hf^0_{i[p]}|$ into linear ones.
\eqref{sub-PWL-fi} evaluates $\hf_{i}$ in the interval $[\hf_{i[p^k]},\hf_{i[p^k+1]}]$.

\indent
The approximate subproblem \eqref{ObRO-sub-PWL} is an LP, which returns $\{\hf^{k+1}_{i[p]}\}_{p=1}^N$ to construct the PWL approximation of the actual optimal function $f_i^{k+1}$.
Also the approximate total deviation of $\hf^{k+1}_i$ from $f_i^0$, say $\hat{\mD}_i$, is obtained as a byproduct, which can be directly used in the master problem to avoid duplicate computation.

\begin{figure}[!h]
  \centering
  \includegraphics[width=3.5in]{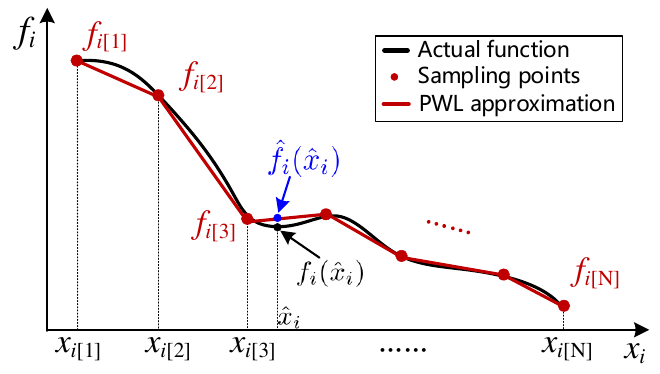}
  \caption{PWL approximation of a function.}
  \label{fig-PWL}
\end{figure}

\indent
With the obtained $\{\hf^{k+1}_{i[p]}\}_{p=1}^N$ and $\hat{\mD}^{k+1}_i$, we construct the PWL approximation of the master problem \eqref{ObRO-master} as follows
\begin{subequations}\label{ObRO-master-PWL}
\begin{align}
      &(\hat{\bm{x}}^{k+1},~\eta^{k+1}) = \notag \\
      &\arg\min~\eta \\
      &\textup{over}~\eta,\hat{\bm{x}},\{\beta_{i[p]}\}_{p=1}^{N-1},\{\alpha_{i[p]}\}_{p=1}^N,i=1,...,n   \notag \\
      &s.t.~\eta\geq \sum\nolimits_{i=1}^n \hf_{i}^l(\hx_i) - \epsilon\hat{\mD}_i^l,~l=0,...,k+1 \\
      &\bm{A}\hat{\bm{x}} \leq \bm{b}  \\
      &\eqref{fihat-beta}-\eqref{fihat-alpha},~i=1,...,n\\
      &\hx_i = \sum\nolimits_{p=1}^{N} \alpha_{i[p]} \hx_{i[p]},~i=1,...,n \\
      &\hf_{i}^l(\hx_i) = \sum_{p=1}^{N} \alpha_{i[p]} \hf^l_{i[p]},~i=1,...,n,~l=0,...,k+1
\end{align}
\end{subequations}
The approximate master problem \eqref{ObRO-master-PWL} is an MILP, which returns $\hat{\bm{x}}^{k+1}$ and $\eta^{k+1}$ as the updated optimal solution and corresponding lower bound $LB$ for the ObRO problem.

\indent
The following result establishes the numerical consistency between the approximate problems \eqref{ObRO-sub-PWL}-\eqref{ObRO-master-PWL} and original problems \eqref{ObRO-sub}-\eqref{ObRO-master}.

\begin{theorem}[Consistency of PWL approximation]\label{thm-consistency}
   Under Assumptions \ref{assump-unique}-\ref{assump-bounded}, there exists a partition of $[x_i^{\min},x_i^{\max}]$, say $\mI_N(x_i)$, with the sampling points $\{x_{i[p]}\}_{p=1}^N$, $i=1,...,n$ such that $\lim_{N\rightarrow\infty} \hat{\bm{x}}^k|_{\mI_N} = \bm{x}^k$ and $\lim_{N\rightarrow\infty} \bm{\hf}^k|_{\mI_N} = \bm{f}^k$, where $\hat{\bm{x}}^k|_{\mI_N}=\{\hx^k_i\}_{i=1}^n,\bm{\hf}^{k}|_{\mI_N}=\{\hf_i^{k}\}_{i=1}^n$, $k=0,1,...$ are the sequence of solutions generated by applying Algorithm~\ref{algObRO} to the approximate master problem \eqref{ObRO-master-PWL} and approximate subproblem \eqref{ObRO-sub-PWL}, and $\bm{x}^k,\bm{f}^k$, $k=0,1,...$ are the sequence of solutions generated by applying Algorithm~\ref{algObRO} to the master problem \eqref{ObRO-master} and subproblem \eqref{ObRO-sub}.
\end{theorem}

\begin{IEEEproof}
   It can be proved using the arbitrary precision of PWL approximation.
   The details are given in Appendix B.
\end{IEEEproof}

Theorem \ref{thm-consistency} ensures that the PWL technique provides appropriate finite-dimensional approximations for the original master problem \eqref{ObRO-master} and subproblem \eqref{ObRO-sub}.
Given a partition of $[x_i^{\min},x_i^{\max}]$ with sufficiently many segments, we can execute Algorithm~\ref{algObRO} on the approximate master problem \eqref{ObRO-master-PWL} and subproblem \eqref{ObRO-sub-PWL} to obtain quality numerical solutions that are sufficiently close to the actual solutions.

\begin{remark}
   The PWL technique provides a feasible numerical solver for ObRO with mathematical rigor.
   But it still may suffer from the dimension curse by directly solving the involved MILP problems in case of a dense segmentation of the interval $[x_i^{\min},x_i^{\max}]$.
   In practice, heuristics or machine learning methods \cite{dai2017learning,bengio2021machine} can be integrated into the MILP solver to speed up the solution process and reach a tradeoff between precision and efficiency.
\end{remark}

\section{Application in BESS charging scheduling}\label{seccase}
\subsection{Battery degradation model}
Battery lifetime degradation is governed by several factors.
For analytical tractability, we focus on the relationship between battery degradation and DoD, defined as the ratio of energy exchanged during a single charge or discharge event to the nominal battery capacity.
We adopt the battery aging reference curve below which is fitted by the least-square interpolation of experiment data~\cite{kong2024lithium}
\begin{equation}\label{batdegrade}
\begin{split}
    f_i^0(P_{bi}) = 9.62|P_{bi}|\Delta t/E_{bi}^{\max} - 4.7(|P_{bi}|\Delta t/E_{bi}^{\max})^2
\end{split}
\end{equation}
where $\Delta t, P_{bi}$ respectively denote the duration of timeslot and charging power during timeslot $t$, $E_{bi}^{\max}$ denotes the battery capacity, and DoD is given by the term $|P_{bi}|\Delta t/E_{bi}^{\max}$.

\subsection{Degradation-aware BESS charging scheduling}
Integrating the empirical degradation model \eqref{batdegrade}, we formulate the optimal BESS charging scheduling problem for a power distribution network as follows
\begin{subequations}\label{ObRO-bat}
\begin{align}
      \min_{P_{bi}^t}&\max_{f_i(\cdot)}~\sum_{i\in\mN}\sum_{t\in\mT}w_v|V_i^t-1|+\sum_{i\in\mN_{bat}}\sum_{t\in\mT} f_i(P_{bi}^t) \notag \\
      &~~~~~~~-\sum_{i\in\mN_{bat}}\epsilon\mD_i(f_i) \label{ObRO-bat-obj} \\
      s.t.~& V_i^t=V_{s}+\sum\nolimits_{j\in\mN}R_{ij}(P_{Wj}^t-P_{bj}^t-P_{Lj}^t) - X_{ij}Q_{Lj}^t, \notag\\
      &~~~~~~~~~~~~~~~~~~~~~~~~~~~~~~~~i\in\mN,t\in\mT   \label{ObRO-bat-pf} \\
      & V_i^{\min} \leq V_i^t \leq  V_i^{\max},~i\in\mN,t\in\mT  \label{ObRO-bat-V}\\
      & P_{bi}^t = 0,~i\in\mN\backslash\mN_{bat},t\in\mT \label{ObRO-bat-Pt}\\
      & P_{bi}^{\min} \leq  P_{bi}^t \leq P_{bi}^{\max},~i\in\mN_{bat},t\in\mT \label{ObRO-bat-Pb}\\
      & 0 \leq E_{bi}^0+\sum_{k=1}^t P_{bi}^k\Delta t \leq E_{bi}^{\max},~i\in\mN_{bat},t\in\mT \label{ObRO-bat-Eb} \\
      &\|f_i - f_i^0\| \leq \delta_i^{\max},~i\in\mN_{bat}  \label{ObRO-bat-fi1}\\
      &\mD_i(f_i) \leq d_i^{\max},~i\in\mN_{bat}  \\
      &\Big|\frac{f_i(x_i^{a})-f_i(x_i^{b})}{f_i^0(x_i^{a})-f_i^0(x_i^{b})}\Big| \leq L_i,\forall x_i^{a},x_i^{b}\in[P_{bi}^{\min},P_{bi}^{\max}] \label{ObRO-bat-fi3}
\end{align}
\end{subequations}
where $\mN,\mN_{bat},\mT$ denotes the set of all non-substation nodes in the network, set of nodes installing BESSs and set of timeslots of interest;
$P_{Wi}^t,P_{Li}^t,Q_{Li}^t$ denotes the renewable power generation, active load and reactive load at node $i$ during timeslot $t$ that are known parameters;
$V_{s}$ denotes the substation voltage that is a predefined constant;
$P_{bi}^t$ denotes the BESS power output at node $i$ during timeslot $t$ ($P_{bi}^t>0$/$P_{bi}^t<0$ means the BESS is getting charged from the network/discharging to the network);
$V_i^t$ denotes the voltage at node $i$ during timeslot $t$ that is determined by \eqref{ObRO-bat-pf}.
The objective \eqref{ObRO-bat-obj} is a weighted sum of voltage deviations from the rated value, battery degradation and small penalty terms $\mD_i$;
\eqref{ObRO-bat-pf} is the linear DistFlow equation where $R_{ij},X_{ij}$ are coefficients determined by the distribution network topology~\cite{farivar2013cdc};
\eqref{ObRO-bat-V} refers to the voltage limits;
\eqref{ObRO-bat-Pb}-\eqref{ObRO-bat-Eb} refer to the BESS power and state of charge limits;
\eqref{ObRO-bat-fi1}-\eqref{ObRO-bat-fi3} regulate the possible deviation of the actual battery degradation model from the reference one in \eqref{batdegrade}.
Overall, problem \eqref{ObRO-bat} coordinates the charging strategies of multiple batteries in the distribution network to minimize the total voltage deviation and battery degradation under the worst-case degradation curve for each battery.

\subsection{Numerical test}
The test is carried out on an 8-node distribution system shown in Fig. \ref{fig8node}, where nodes 3 and 5 install photovoltaic panels and nodes 2 and 6 install BESS devices.
We refer to \cite{case8data} for the detailed system information including the peak load, peak photovoltaic generation and network parameters.
In addition, we assume the load and photovoltaic generation at each node follow the daily curves in Fig. \ref{figloadcurve}.

\indent
To construct the model \eqref{ObRO-bat}, we adopt the following parameters: $\mN=\{1,...,8\}$, $\mT=\{1,...,24\}$ (hourly scheduling), $w_v=10$, $\epsilon=0.1$, $P_{bi}^{\max}=0.04$, $E_{bi}^{\max}=0.2$, $E_{bi}^{0}=0$, $V_i^{\min}=0.95$, $V_i^{\max}=1.05$, $\delta_i=0.05$, $d_i^{\max}=10^{-3}$, $L_i=1.5$, $\varepsilon=10^{-2}$ (stop criterion).
Under these settings, Algorithm \ref{algObRO} converges in 79 iterations and the trajectories of $UB$ and $LB$ are shown in Fig. \ref{figUBLB}.
The optimal charging strategies and worst-case degradation functions for BESSs at nodes 2 and 6 are depicted in Fig. \ref{figoptPb} and Fig. \ref{figoptfunc}, respectively.

\indent
As illustrated in Fig. \ref{figoptfunc}, the worst-case degradation function for BESS at node 2 (blue curve) takes larger values than the reference function (black curve) in both medium and high charging power intervals.
On the other hand, the worst-case function for BESS at node 6 (orange curve) takes larger values than the reference function only in low charging power intervals.
To reduce degradation in this worst case, in Fig. \ref{figoptPb} the BESS at node 2 has lower peak charging power than the initial solution (see the blue and pink curves), and the BESS at node 6 tends to take more extreme charging power than the initial solution (see the green and red curves).

\indent
The above test is carried out by evenly segmenting the interval $[0,P_{bi}^{\max}]$ into 20 segments ($\Delta P_{bi}=0.002$). To compare the solutions under different resolutions, we set up the following segmentation schemes.
\begin{itemize}
  \item Scheme 1 (sparse): even segment with $\Delta P_{bi}=0.004$ for the whole interval $[0,P_{bi}^{\max}]$;
  \item Scheme 2 (benchmark): even segment with $\Delta P_{bi}=0.002$ for the whole interval $[0,P_{bi}^{\max}]$;
  \item Scheme 3 (dense): even segment with $\Delta P_{bi}=0.0008$ for the whole interval $[0,P_{bi}^{\max}]$;
  \item Scheme 4 (heterogeneous): even segment with $\Delta P_{bi}=0.0008$ for $[0,P_{bi}^{\max}/2]$, and even segment with $\Delta P_{bi}=0.002$ for $[P_{bi}^{\max}/2,P_{bi}^{\max}]$;
  \item Scheme 5 (parametric uncertainty): optimization under the assumption that the degradation function takes the form below with uncertain parameters $a,b$
  \begin{equation}
  \begin{split}
      f_i^0(P_{bi}) = a|P_{bi}|\Delta t/E_{bi}^{\max} - b(|P_{bi}|\Delta t/E_{bi}^{\max})^2
  \end{split}
  \end{equation}
  where $9\leq a\leq 10$ and $4\leq b \leq 5$.
\end{itemize}

\indent
The worst-case degradation functions in the above five cases are depicted in Fig.~\ref{figoptfuncs2} (for BESS at node 2) and Fig.~\ref{figoptfuncs6} (for BESS at node 6).
The optimization under parametric uncertainty (Scheme 5) has few degrees of freedom, trivially driving $a,b$ to their boundary values ($a=10, b=4$) to overall elevate the function curve.
In contrast, the ObRO solutions manage to capture those local variations in the curve, which are generally consistent with each other.
The curves under Schemes 1-4 all indicate that the functional variation around $P_{bi}=0.02, 0.035$ (or $0.005<P_{bi}<0.015$) will worsen the performance of BESS at node 2 (or node 6).
Moreover, it can be observed from Schemes 1-3 that a denser segmentation leads to richer nuances in the function expression.
The similarity between Schemes 3 and 4 suggests that the computational complexity can be reduced by deploying dense segmentation for subintervals of critical interest and maintaining sparse segmentation elsewhere.

\begin{figure}[!h]
  \centering
  \subfigure[System diagram.]{
  \label{fig8node} 
  \includegraphics[width=2.0in]{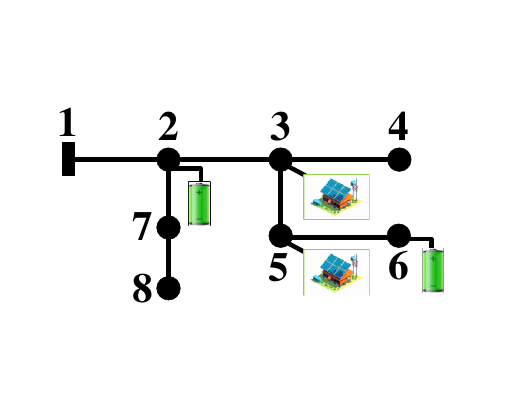}}
  \subfigure[Daily load and PV curves.]{
  \label{figloadcurve} 
  \includegraphics[width=2.7in]{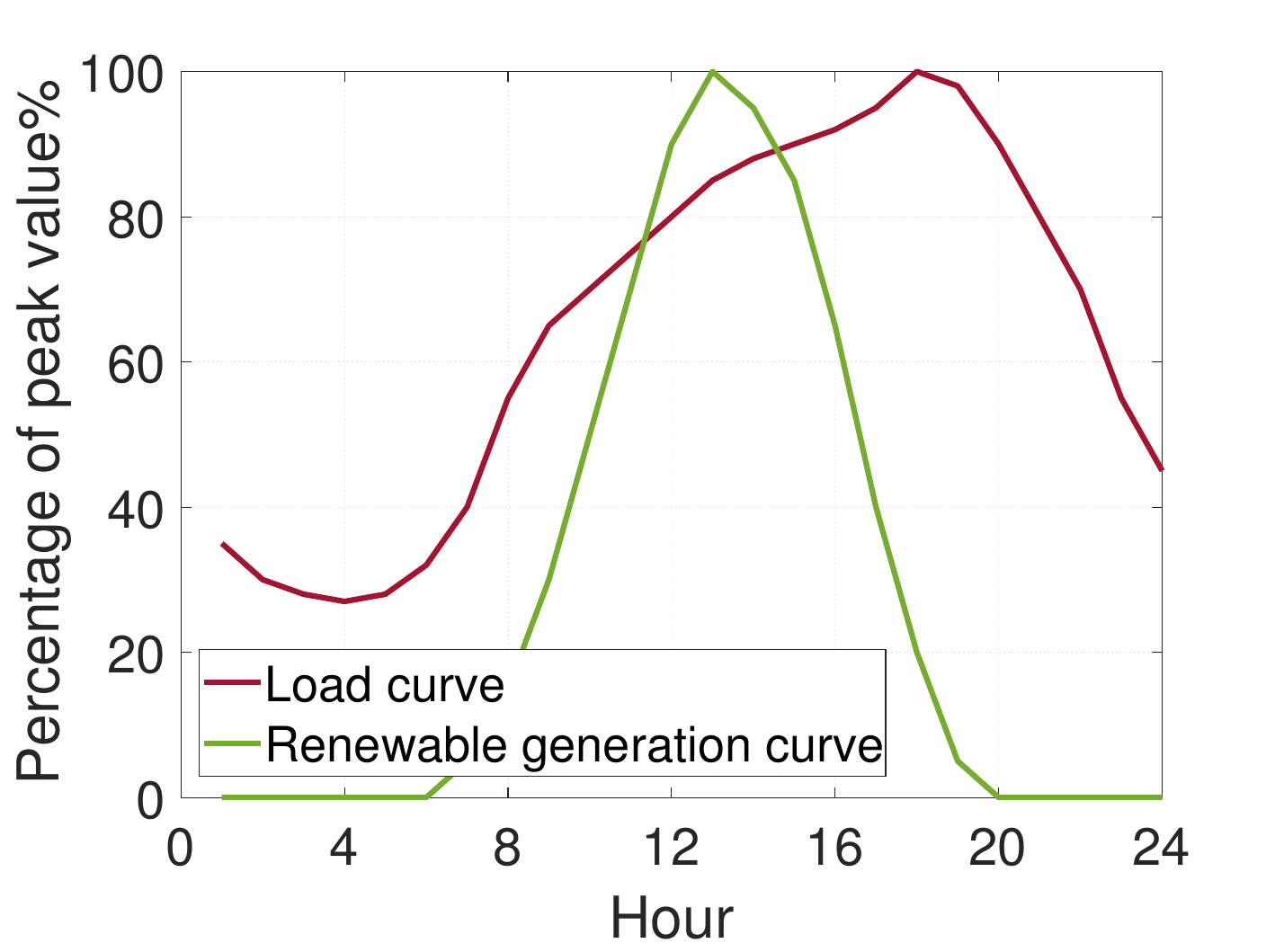}}
  \caption{Information of the 8-node test system.}
  \label{figsysinfo} 
\end{figure}

\begin{figure}[!h]
  \centering
  \subfigure[Iteration process.]{
  \label{figUBLB} 
  \includegraphics[width=2.7in]{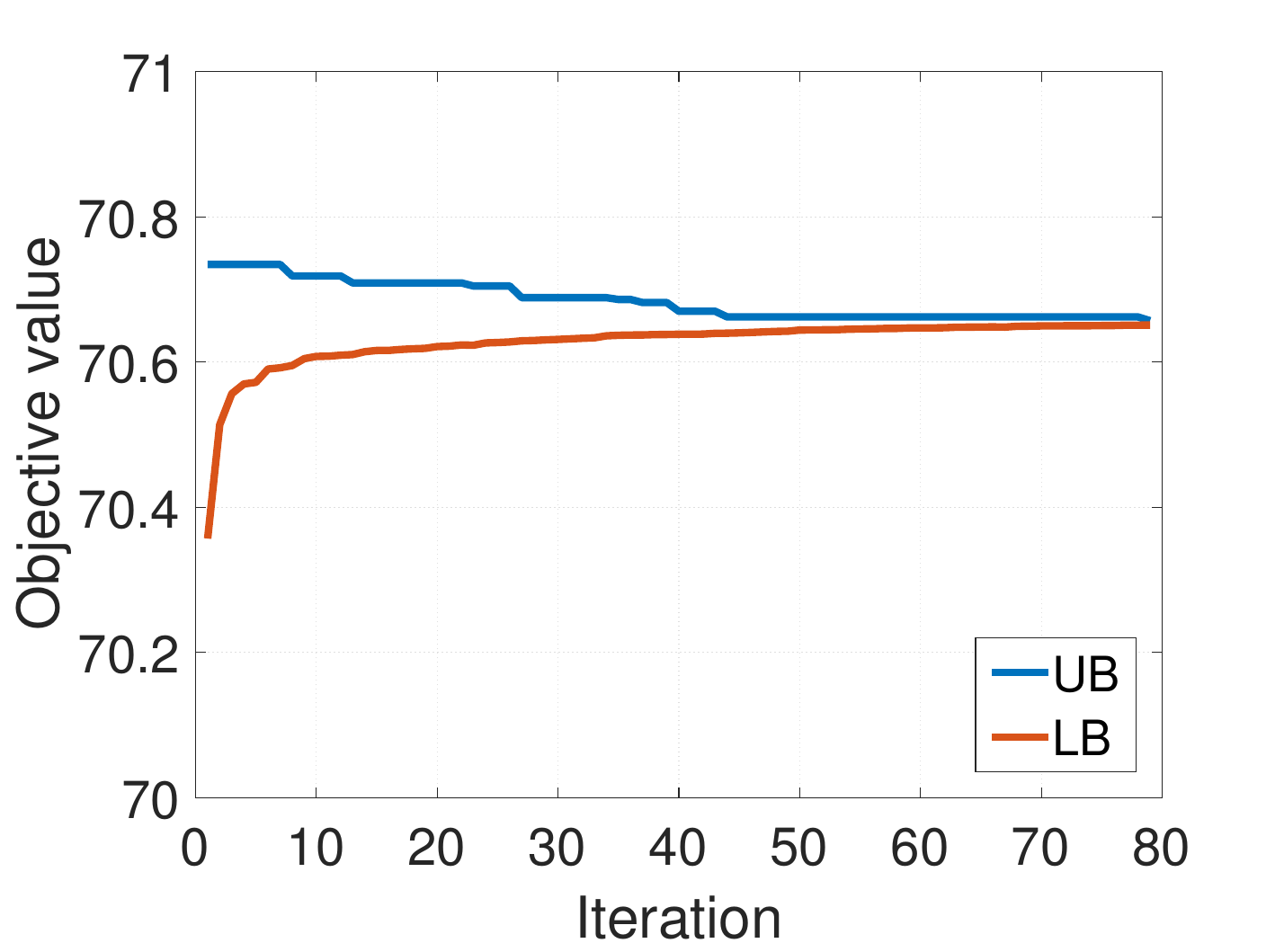}}
  \subfigure[Charging strategies comparison.]{
  \label{figoptPb} 
  \includegraphics[width=2.7in]{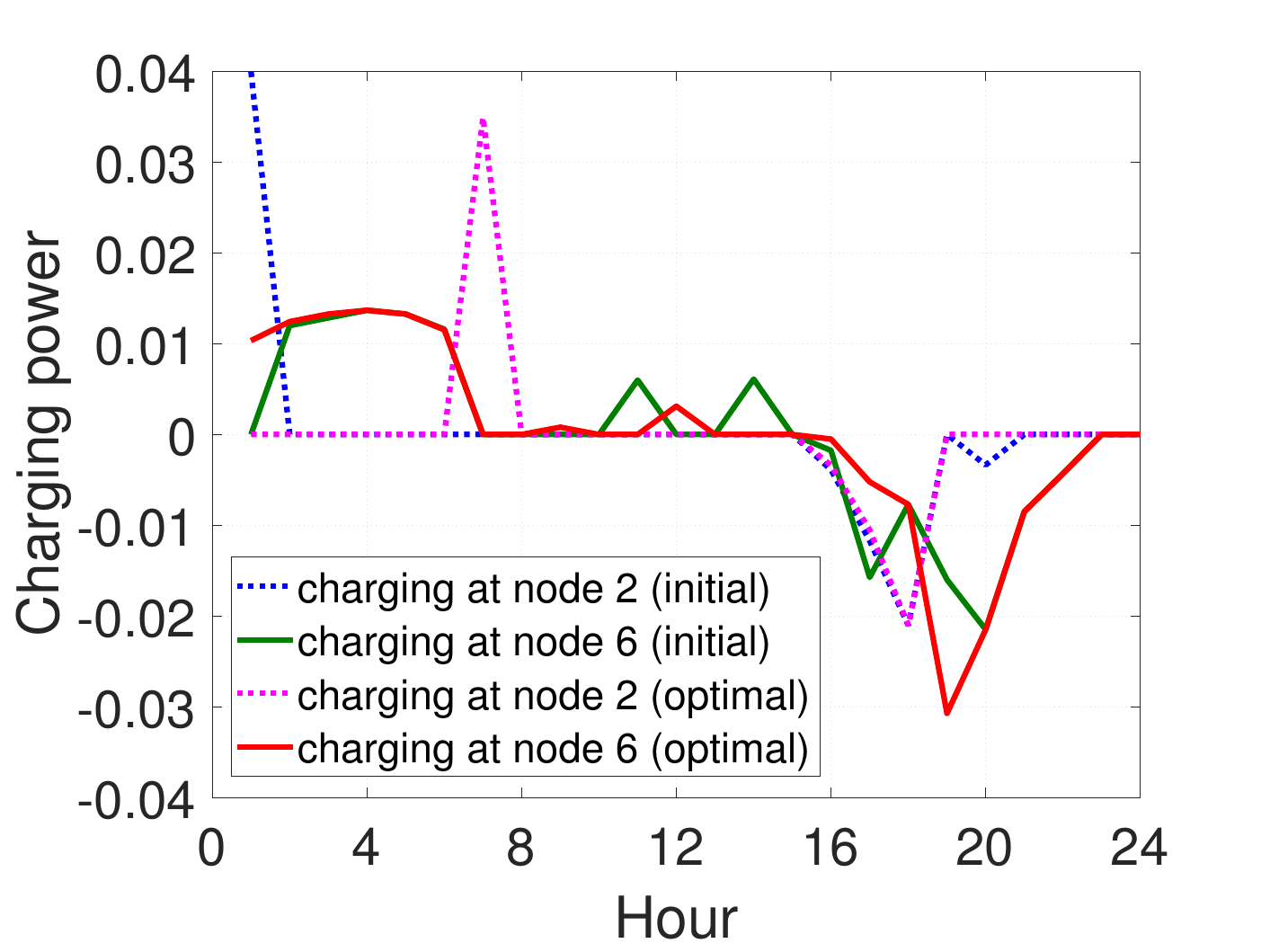}}
  \caption{The solution information}
  \label{figsol} 
\end{figure}

\begin{figure}[!h]
  \centering
  \includegraphics[width=3.5in]{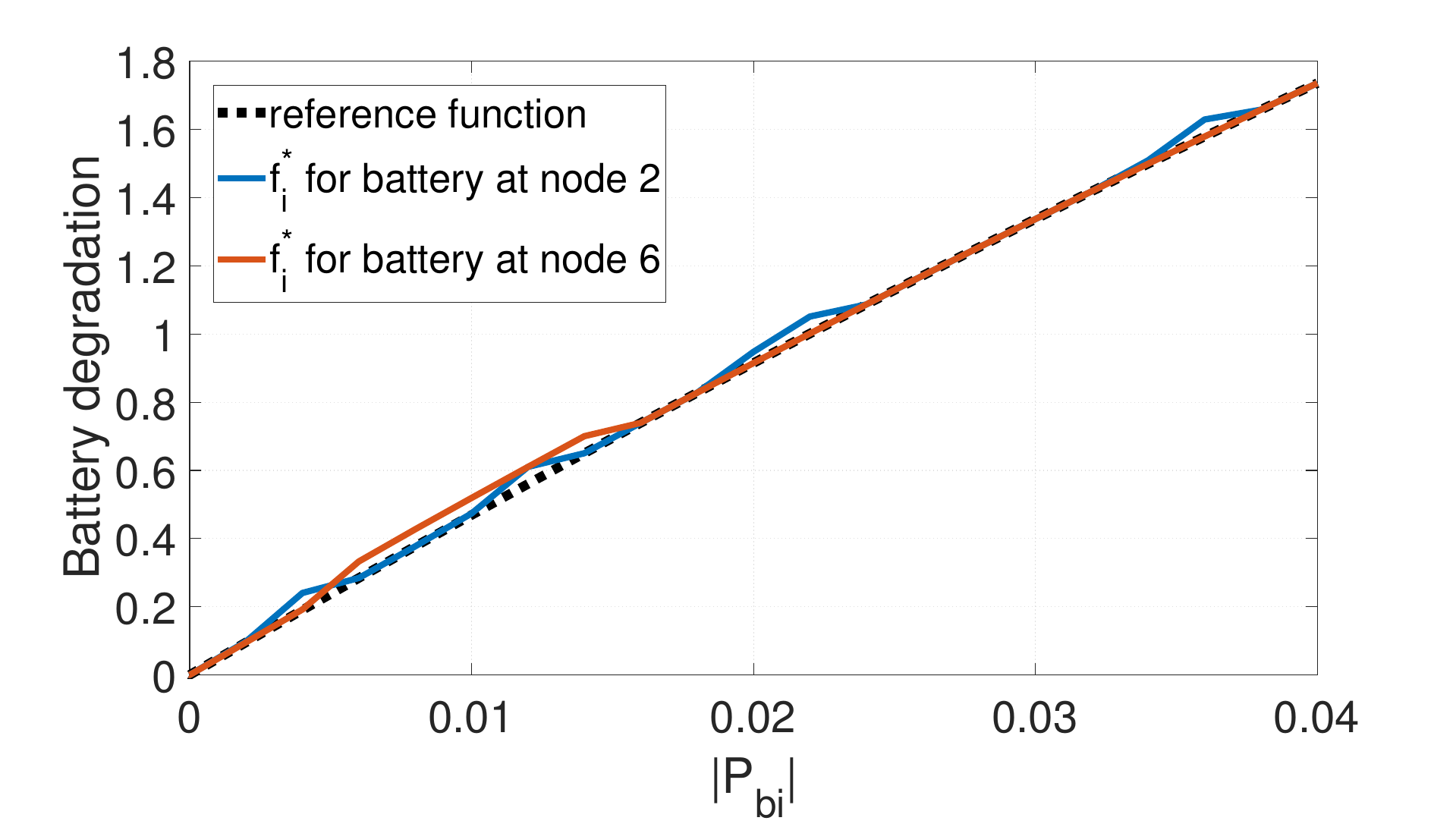}
  \caption{Reference degradation function v.s. worst-case functions.}
  \label{figoptfunc}
\end{figure}

\begin{figure}[!h]
  \centering
  \includegraphics[width=3.5in]{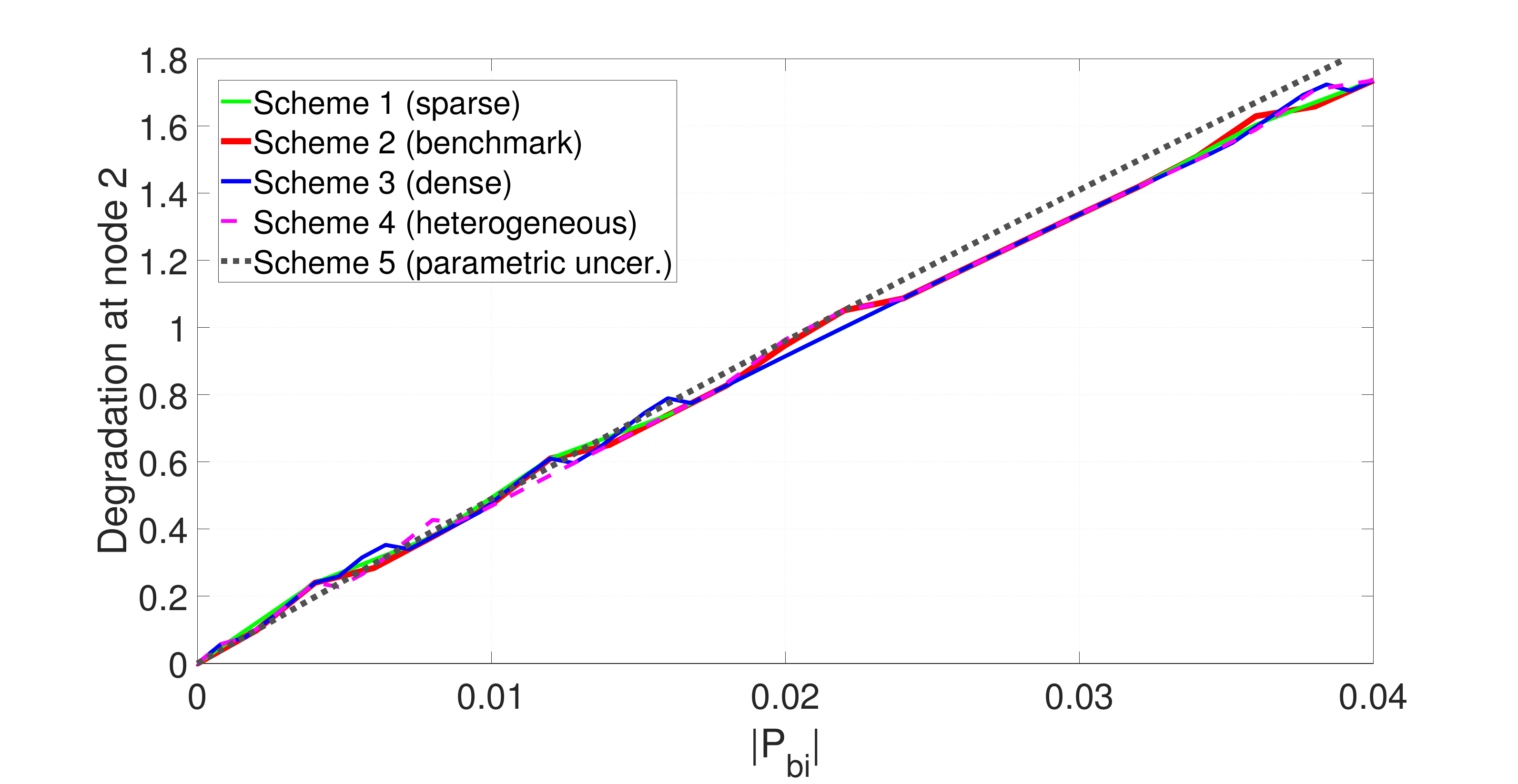}%
  \caption{Worst-case functions under five schemes for BESS at node 2.}
  \label{figoptfuncs2}
\end{figure}

\begin{figure}[!h]
  \centering
  \includegraphics[width=3.5in]{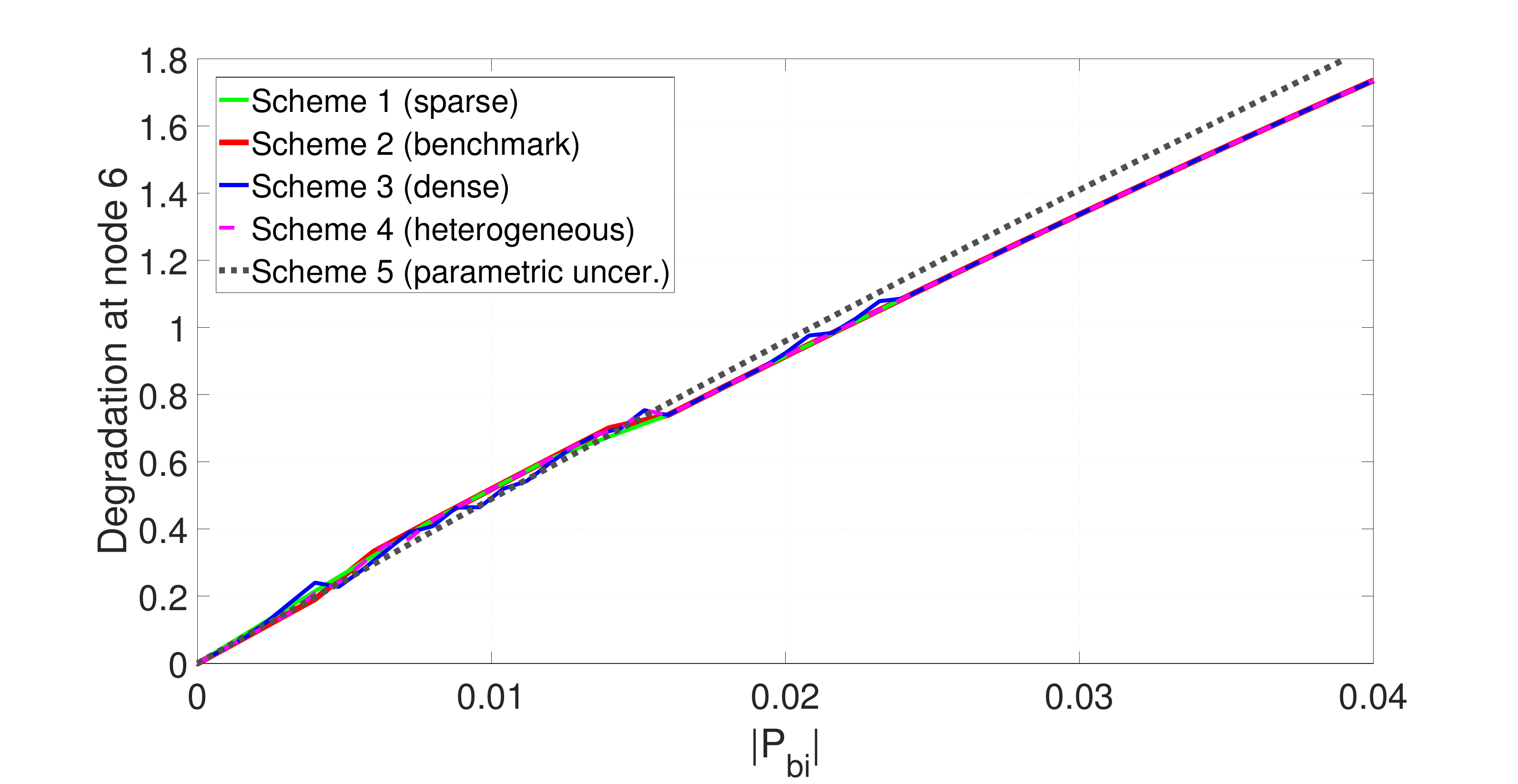}%
  \caption{Worst-case functions under five schemes for BESS at node 6.}
  \label{figoptfuncs6}
\end{figure}

\section{Conclusion}\label{secconclu}
We have formulated the ObRO problem that finds the decision with the minimum cost under the worst-case objective function.
We have designed an alternate iterative algorithm which converges to a semi-global saddle point of the ObRO problem.
A PWL-based numerical solver has also been proposed, which is consistent with the original ObRO problem.
The obtained results have been applied to the degradation-aware BESS charging scheduling in distribution networks.

\section*{Appendix A: Proof of Theorem \ref{thm-converge}}
Let us first present the following three lemmas as a basis.
\begin{lemma}\label{lem-conti}
    Under Assumptions \ref{assump-unique}-\ref{assump-bounded}, the operator $\mT_2\circ\mT_1$ is continuous and uniformly compact.
\end{lemma}

\begin{IEEEproof}
   We first establish the continuity of $\mT_1$ and $\mT_2$.
   Given $\bm{x}^{k}$, let $\bm{f}^{k+1}=\mT_1\bm{x}^k$.
   Consider a sequence $\{\bm{y}_j\}$ that $\lim_{j\rightarrow\infty} \bm{y}_j=\bm{x}^{k}$, and let $\bm{f}^{y_j}=\arg\max\nolimits_{\bm{f}\in\Nbmf0}V(\bm{f},\bm{y}_{j})$.
   Then, $\lim_{j\rightarrow\infty} f\bm{f}^{y_j}=\bm{f}^{k+1}$ must hold due to the uniqueness of optimal solution of the subproblem \eqref{ObRO-sub} under Assumption \ref{assump-unique}. It implies that $\mT_1$ is continuous.
   Following a similar idea, we can show that $\mT_2$ is also continuous and hence the composite operator $\mT_2\circ\mT_1$ is continuous.

   Next we show the uniform compactness of $\mT_2\circ\mT_1$.
   By Assumption \ref{assump-bounded} and constraint \eqref{fneighbor}, it can be easily proved that
    \begin{itemize}
      \item $\Nf0$ is uniformly bounded, i.e., there exists $M_i=M+\delta_i^{\max}$ such that $\| f_i \|\leq M_i$, $\forall f_i\in\Nf0$.
      \item $\Nf0$ is equicontinuous, i.e., $\forall \varepsilon>0$, there exists $\delta=\frac{\varepsilon}{L_i M}$ (independent of $f_i$) such that $\forall x_i^a,x_i^b\in[x_i^{\min},x_i^{\max}]$ with $|x_i^a - x_i^b|<\delta$ we have $|f_i(x_i^a)-f_i(x_i^b)| < L_i M |x_i^a - x_i^b|< \varepsilon$, $\forall f_i\in\Nf0$.
      \item $\Nf0$ is closed.
    \end{itemize}
    By Arzela-Ascoli Theorem \cite{rudin1991functional}, we conclude that $\Nf0$ is a compact set since it is uniformly bounded, equicontinuous and bounded.
    In addition, we conclude that the finite-dimensional set $\mS_x$ is also compact due to its closedness and boundness (it includes $x_i^{\min}\leq x_i\leq x_i^{\max}$).
    So, $\mT_2\circ\mT_1$ is a uniformly compact operator as there exist compact sets $\mS_x$ and $\Nbmf0$ such that $\mT_1(\bm{x})\in \Nbmf0$ and $\mT_2\circ\mT_1(\bm{x})\in \mS_x$, $\forall\bm{x}\in \mS_x$.
\end{IEEEproof}

\begin{lemma}\label{lem-monotonic}
    Under Assumption \ref{assump-unique}, the operator $\mT_2\circ\mT_1$ is strictly monotonic with respect to $UB-LB$ in Algorithm~\ref{algObRO}.
\end{lemma}

\begin{IEEEproof}
   It trivially follows that $UB$ is non-increasing during the iteration in Algorithm~\ref{algObRO}. The subsequent proof demonstrates that $LB$ is strictly increasing.

   Suppose we have solution $\bm{x}^{k}$ at hand, which is obtained by solving the master problem below in the previous iteration
   \begin{subequations}\label{masterfk}
   \begin{align}
      (\bm{x}^{k},\eta^{k})=\arg\min_{\bm{x}\in\mS_x,\eta}~&\eta \\
      s.t.~&\eta\geq V(\bm{f}^l,\bm{x}),~l=0,...,k.
   \end{align}
   \end{subequations}
   with the following properties
   \begin{subequations}
   \begin{align}
      &\eta^{k}=\max\{V(\bm{f}^l,\bm{x}^{k}),l=0,1,...,k\} \label{etakm1}\\
      &\max\{V(\bm{f}^l,\bm{x}),l=0,1,...,k\}> \notag \\
      &~~~~~~\max\{V(\bm{f}^l,\bm{x}^{k}),l=0,1,...,k\},\forall \bm{x}\neq \bm{x}^{k}  \label{xkm1}
   \end{align}
   \end{subequations}
   due to the uniqueness of master problem optimum under Assumption~\ref{assump-unique}.

   Given $\bm{x}^{k}$, solving \eqref{ObRO-sub} gives $\bm{f}^{k+1}$.
   The uniqueness of subproblem optimum under Assumption~\ref{assump-unique} gives
   \begin{equation}\label{lastsub}
   \begin{split}
       V(\bm{f}^{k+1},\bm{x}^{k}) > V(\bm{f},\bm{x}^{k}),~\forall \bm{f}\neq \bm{f}^{k+1}.
   \end{split}
   \end{equation}
   Then, $(\bm{x}^{k+1},\eta^{k+1})$ is obtained by solving
   \begin{subequations}\label{masterfk}
   \begin{align}
      (\bm{x}^{k+1},\eta^{k+1})=\arg\min_{\bm{x}\in\mS_x,\eta}~&\eta \\
      s.t.~&\eta\geq V(\bm{f}^l,\bm{x}),~l=0,...,k \label{eta-fi} \\
      &\eta\geq V(\bm{f}^{k+1},\bm{x}). \label{eta-fk}
   \end{align}
   \end{subequations}
   If $\bm{f}^{k+1}=\bm{f}^l$, $l=0,...,k$, it directly implies $\bm{x}^{k+1}=\bm{x}^{k}$ and $\bm{x}^{k}$ is a fixed point of $\mT_2\circ\mT_1$, rendering the definition of strictly monotonic operators inapplicable.
   Thus, we assume $\bm{f}^{k+1}\neq \bm{f}^l$, $l=0,...,k$, and examine the following two distinct cases:
   \begin{enumerate}
     \item If $\bm{x}^{k+1}=\bm{x}^{k}$, then by \eqref{etakm1} and \eqref{lastsub} we have
     \begin{equation}
     \begin{split}
       \eta^{k+1}&=\max\{V(\bm{f}^l,\bm{x}^{k+1}),l=0,...,k+1\}\\
       &=\max\{V(\bm{f}^l,\bm{x}^{k}),l=0,...,k+1\}\\
       &=V(\bm{f}^{k+1},\bm{x}^{k})\\
       &>\max\{V(\bm{f}^l,\bm{x}^{k}),l=0,...,k\}=\eta^{k}.
     \end{split}
     \end{equation}
     \item If $\bm{x}^{k+1}\neq \bm{x}^{k}$, then \eqref{xkm1} and \eqref{lastsub} we have
     \begin{equation}
     \begin{split}
       \eta^{k+1}&=\max\{V(\bm{f}^l,\bm{x}^{k+1}),l=0,...,k+1\}\\
       &\geq \max\{V(\bm{f}^l,\bm{x}^{k+1}),l=0,...,k\}\\
       &>\max\{V(\bm{f}^l,\bm{x}^{k}),l=0,...,k\}=\eta^{k}.
     \end{split}
     \end{equation}
   \end{enumerate}
   It implies that $LB$ is strictly increasing in both cases and hence $\mT_2\circ\mT_1$ is strictly monotonic with respect to $UB-LB$.
\end{IEEEproof}

\begin{lemma}[\cite{meyer1976sufficient}]\label{lem-meyer}
    Let $\mT: Y_1\rightarrow Y_1$ be an operator that is continuous, uniformly compact, and strictly monotonic on $Y_1$ with respect to the functional $g$.
    If the number of fixed points of $\mT$ having any given value of $g$ is finite, the sequence $\{y_i\}_{i=1}^{\infty}$ generated by $y_{i+1}=\mT(y_i)$ converges to a fixed point of $\mT$.
\end{lemma}

\textit{Proof of Theorem \ref{thm-converge}:}
   It can be concluded by combining Lemmas~\ref{lem-conti}-\ref{lem-meyer}.  \hfill $\blacksquare$

\section*{Appendix B: Proof of Theorem \ref{thm-consistency}}
\begin{lemma}[\cite{lin1992canonical}]\label{lem-PWL}
    Given a function $f_i\in C[x_i^{\min},x_i^{\max}]$ and any $\varepsilon>0$, there exists an integer $N>0$ and partition $\mI_N(x_i)$ such that $\hf_{i}$ given by \eqref{fihat} satisfies $\|\hf_{i} - f_i\|<\varepsilon$.
\end{lemma}

\textit{Proof of Theorem \ref{thm-consistency}:}
   First, we will show that
   \begin{subequations}
   \begin{align}
      \lim\nolimits_{N\rightarrow\infty} \hat{\bm{x}}^0|_{\mI_N} &= \bm{x}^0 \label{consis-x0} \\
      \lim\nolimits_{N\rightarrow\infty} \bm{\hf}^1|_{\mI_N} &= \bm{f}^1   \label{consis-f1}  \\
      \lim\nolimits_{N\rightarrow\infty} \hat{\bm{x}}^1|_{\mI_N} &= \bm{x}^1  \label{consis-x1}
   \end{align}
   \end{subequations}
   which refers to the initialization step and first iteration in Algorithm~\ref{algObRO}.

   \indent
   During the initialization phase, we obtain $\hat{\bm{x}}^0$ by solving \eqref{ObRO-master-PWL} with the approximate reference function $\bm{\hf}^0$ (i.e., the PWL version of $\bm{f}^0$ at the sampling points $\{x_{i[p]}\}_{p=1}^N$).
   Meanwhile, we obtain $\bm{x}^0$ by solving \eqref{ObRO-master} with the actual reference function $\bm{f}^0$.
   By Lemma \ref{lem-PWL} and completeness of the function space $C[x_i^{\min},x_i^{\max}]$, there exists a partition $\mI_N(x_i)$ such that the sequence $\bm{\hf}^0|_{\mI_N}$, $N\in[N_1,\infty)$ forms a Cauchy sequence and
   \begin{equation}\label{inf-x0}
   \begin{split}
      \lim\nolimits_{N\rightarrow\infty} \bm{\hf}^0|_{\mI_N} = \bm{f}^0.
   \end{split}
   \end{equation}
   In addition, Lemma \ref{lem-conti} already implies that the operators $\mT_1,\mT_2$ are continuous.
   Then, \eqref{consis-x0} follows from the relation $\hat{\bm{x}}^0|_{\mI_N}=\mT_2(\bm{\hf}^0|_{\mI_N})$, $\bm{x}^0=\mT_2(\bm{f}^0)$ and continuity of operator $\mT_2$.

   \indent
   Consider the approximate subproblem \eqref{ObRO-sub-PWL} in the first iteration in Algorithm~\ref{algObRO}.
   Similar to the derivation of \eqref{inf-x0}, we conclude that constraints \eqref{sub-PWL-delta}-\eqref{sub-PWL-lip} approach \eqref{ObRO-f0}-\eqref{ObRO-lip} when $N\rightarrow\infty$.
   Together with \eqref{consis-x0} and the continuity of operator $\mT_1$, it leads to the satisfaction of \eqref{consis-f1}.
   Furthermore, \eqref{consis-x1} can be proved in a similar way using \eqref{consis-f1}, \eqref{inf-x0} and continuity of operator $\mT_2$.

   \indent
   By iteratively applying \eqref{consis-x0}-\eqref{consis-x1} alongside the established continuity of $\mT_1,\mT_2$, we can similarly prove the consistency between the solutions of the approximate and original problems across all subsequent iterations of Algorithm~\ref{algObRO}. \hfill $\blacksquare$

\ifCLASSOPTIONcaptionsoff
  \newpage
\fi

{\footnotesize
\bibliographystyle{IEEEtran}
\bibliography{IEEEabrv,objuncertain}

\begin{thebibliography}{10}
\providecommand{\url}[1]{#1}
\csname url@samestyle\endcsname
\providecommand{\newblock}{\relax}
\providecommand{\bibinfo}[2]{#2}
\providecommand{\BIBentrySTDinterwordspacing}{\spaceskip=0pt\relax}
\providecommand{\BIBentryALTinterwordstretchfactor}{4}
\providecommand{\BIBentryALTinterwordspacing}{\spaceskip=\fontdimen2\font plus
\BIBentryALTinterwordstretchfactor\fontdimen3\font minus
  \fontdimen4\font\relax}
\providecommand{\BIBforeignlanguage}[2]{{%
\expandafter\ifx\csname l@#1\endcsname\relax
\typeout{** WARNING: IEEEtran.bst: No hyphenation pattern has been}%
\typeout{** loaded for the language `#1'. Using the pattern for}%
\typeout{** the default language instead.}%
\else
\language=\csname l@#1\endcsname
\fi
#2}}
\providecommand{\BIBdecl}{\relax}
\BIBdecl

\bibitem{ben2009robust}
A.~Ben-Tal, L.~El~Ghaoui, and A.~Nemirovski, \emph{Robust Optimization}.\hskip
  1em plus 0.5em minus 0.4em\relax Princeton University Press, 2009.

\bibitem{drnach2021robust}
L.~Drnach and Y.~Zhao, ``Robust trajectory optimization over uncertain terrain
  with stochastic complementarity,'' \emph{IEEE Robotics and Automation
  Letters}, vol.~6, no.~2, pp. 1168--1175, 2021.

\bibitem{ben2011robust}
A.~Ben-Tal, B.~Do~Chung, S.~R. Mandala, and T.~Yao, ``Robust optimization for
  emergency logistics planning: Risk mitigation in humanitarian relief supply
  chains,'' \emph{Transportation Research Part {B}: Methodological}, vol.~45,
  no.~8, pp. 1177--1189, 2011.

\bibitem{sun2021robust}
X.~A. Sun and A.~J. Conejo, \emph{Robust Optimization in Electric Energy
  Systems}.\hskip 1em plus 0.5em minus 0.4em\relax Springer, 2021, vol. 313.

\bibitem{sinha2018certifying}
A.~Sinha, H.~Namkoong, and J.~Duchi, ``Certifying some distributional
  robustness with principled adversarial training,'' in \emph{International
  Conference on Learning Representations}, 2018.

\bibitem{zeng2013solving}
B.~Zeng and L.~Zhao, ``Solving two-stage robust optimization problems using a
  column-and-constraint generation method,'' \emph{Operations Research
  Letters}, vol.~41, no.~5, pp. 457--461, 2013.

\bibitem{delage2010distributionally}
E.~Delage and Y.~Ye, ``Distributionally robust optimization under moment
  uncertainty with application to data-driven problems,'' \emph{Operations
  Research}, vol.~58, no.~3, pp. 595--612, 2010.

\bibitem{zhao2018data}
C.~Zhao and Y.~Guan, ``Data-driven risk-averse stochastic optimization with
  wasserstein metric,'' \emph{Operations Research Letters}, vol.~46, no.~2, pp.
  262--267, 2018.

\bibitem{rahimian2022frameworks}
H.~Rahimian and S.~Mehrotra, ``Frameworks and results in distributionally
  robust optimization,'' \emph{Open Journal of Mathematical Optimization},
  vol.~3, pp. 1--85, 2022.

\bibitem{kuhn2025distributionally}
D.~Kuhn, S.~Shafiee, and W.~Wiesemann, ``Distributionally robust
  optimization,'' \emph{Acta Numerica}, vol.~34, pp. 579--804, 2025.

\bibitem{maccheroni2002maxmin}
F.~Maccheroni, ``Maxmin under risk,'' \emph{Economic Theory}, vol.~19, no.~4,
  pp. 823--831, 2002.

\bibitem{armbruster2015decision}
B.~Armbruster and E.~Delage, ``Decision making under uncertainty when
  preference information is incomplete,'' \emph{Management Science}, vol.~61,
  no.~1, pp. 111--128, 2015.

\bibitem{hu2017optimization}
J.~Hu and G.~Stepanyan, ``Optimization with reference-based robust preference
  constraints,'' \emph{{SIAM} Journal on Optimization}, vol.~27, no.~4, pp.
  2230--2257, 2017.

\bibitem{haskell2022preference}
W.~B. Haskell, H.~Xu, and W.~Huang, ``Preference robust optimization for choice
  functions on the space of cdfs,'' \emph{{SIAM} Journal on Optimization},
  vol.~32, no.~2, pp. 1446--1470, 2022.

\bibitem{ju2018two}
C.~Ju, P.~Wang, L.~Goel, and Y.~Xu, ``A two-layer energy management system for
  microgrids with hybrid energy storage considering degradation costs,''
  \emph{{IEEE} Trans. Smart Grid}, vol.~9, no.~6, pp. 6047--6057, 2018.

\bibitem{kong2024lithium}
L.~Kong, S.~Fang, T.~Niu, G.~Chen, L.~Yang, and R.~Liao, ``A lithium-ion
  battery capacity degradation correction method for off-design cycling
  conditions,'' \emph{{IEEE} Trans. Ind. Appl.}, vol.~60, no.~4, pp.
  5778--5789, 2024.

\bibitem{wang2026model}
G.~Wang, H.~Shu, B.~Shi, J.~Li, H.~Liu, and S.~Lei, ``Model-data hybrid driven
  {SOC} estimation for energy storage batteries using {AFF-RLS and
  BiLSTM-Transformer},'' \emph{Protection and Control of Modern Power Systems},
  vol.~11, no.~2, pp. 89--106, 2026.

\bibitem{terkelsen1972some}
F.~Terkelsen, ``Some minimax theorems,'' \emph{Mathematica Scandinavica},
  vol.~31, no.~2, pp. 405--413, 1972.

\bibitem{dempe2015bilevel}
S.~Dempe, V.~Kalashnikov, G.~A. P{\'e}rez-Vald{\'e}s, and N.~Kalashnykova,
  \emph{Bilevel Programming Problems: Theory, Algorithms and Applications to
  Energy Networks}.\hskip 1em plus 0.5em minus 0.4em\relax Springer, 2015.

\bibitem{simchi2019constraint}
D.~Simchi-Levi, H.~Wang, and Y.~Wei, ``Constraint generation for two-stage
  robust network flow problems,'' \emph{{INFORMS} Journal on Optimization},
  vol.~1, no.~1, pp. 49--70, 2019.

\bibitem{song2025adaptive}
Y.~Song, T.~Liu, and G.~Li, ``Adaptive robust optimal control of constrained
  continuous-time linear systems: A functional constraint generation
  approach,'' \emph{{IEEE} Trans. Autom. Control}, vol.~70, no.~2, pp.
  1312--1319, 2025.

\bibitem{lorca2016multistage}
{\'A}.~Lorca, X.~A. Sun, E.~Litvinov, and T.~Zheng, ``Multistage adaptive
  robust optimization for the unit commitment problem,'' \emph{Operations
  Research}, vol.~64, no.~1, pp. 32--51, 2016.

\bibitem{meyer1976sufficient}
R.~R. Meyer, ``Sufficient conditions for the convergence of monotonic
  mathematical programming algorithms,'' \emph{Journal of Computer and System
  Sciences}, vol.~12, no.~1, pp. 108--121, 1976.

\bibitem{gan2015exact}
L.~Gan, N.~Li, U.~Topcu, and S.~H. Low, ``Exact convex relaxation of optimal
  power flow in radial networks,'' \emph{{IEEE} Trans. Autom. Control},
  vol.~60, no.~1, pp. 72--87, 2015.

\bibitem{kocuk2016strong}
B.~Kocuk, S.~S. Dey, and X.~A. Sun, ``Strong {SOCP} relaxations for the optimal
  power flow problem,'' \emph{Operations Research}, vol.~64, no.~6, pp.
  1177--1196, 2016.

\bibitem{fooladivanda2018energy}
D.~Fooladivanda and J.~A. Taylor, ``Energy-optimal pump scheduling and water
  flow,'' \emph{{IEEE} Trans. Control Netw. Syst.}, vol.~5, no.~3, pp.
  1016--1026, 2018.

\bibitem{zheng2019online}
Y.~Zheng, Y.~Song, D.~J. Hill, and K.~Meng, ``Online distributed {MPC}-based
  optimal scheduling for {EV} charging stations in distribution systems,''
  \emph{{IEEE} Trans. Ind. Informat.}, vol.~15, no.~2, pp. 638--649, 2019.

\bibitem{singh2019natural}
M.~K. Singh and V.~Kekatos, ``Natural gas flow equations: Uniqueness and an
  {MI-SOCP} solver,'' in \emph{Proc. Amer. Control Conf.}, 2019, pp.
  2114--2120.

\bibitem{dai2017learning}
H.~Dai, E.~B. Khalil, Y.~Zhang, B.~Dilkina, and L.~Song, ``Learning
  combinatorial optimization algorithms over graphs,'' in \emph{Proc. NeurIPS},
  2017, pp. 6351--6361.

\bibitem{bengio2021machine}
Y.~Bengio, A.~Lodi, and A.~Prouvost, ``Machine learning for combinatorial
  optimization: A methodological tour d'horizon,'' \emph{European Journal of
  Operational Research}, vol. 290, no.~2, pp. 405--421, 2021.

\bibitem{farivar2013cdc}
M.~Farivar, L.~Chen, and S.~Low, ``Equilibrium and dynamics of local voltage
  control in distribution systems,'' in \emph{Proc. {IEEE} Conf. Dec. Control},
  2013, pp. 4329--4334.

\bibitem{case8data}
\BIBentryALTinterwordspacing
``case8adn.m.'' [Online]. Available:
  \url{https://drive.google.com/file/d/1cyqHKITynbqEjdwFYNWN1APS6a0VQWMe/view?usp=sharing}
\BIBentrySTDinterwordspacing

\bibitem{rudin1991functional}
W.~Rudin, \emph{Functional Analysis}.\hskip 1em plus 0.5em minus 0.4em\relax
  McGraw-Hill, 1991.

\bibitem{lin1992canonical}
J.-N. Lin and R.~Unbehauen, ``Canonical piecewise-linear approximations,''
  \emph{{IEEE} Trans. Circuits Syst. {I}}, vol.~39, no.~8, pp. 697--699, 1992.

\end{thebibliography}

}




\end{document}